\documentclass[11pt]{amsart}
\usepackage{amsmath}
\usepackage{amssymb}
\usepackage{amsthm}
\usepackage{array}
\usepackage{xy}
\usepackage[pdftex]{graphicx}
\usepackage{hyperref}
\usepackage{color}
\usepackage{transparent}
\usepackage{latexsym}

\theoremstyle{plain}
\newtheorem{theorem}{Theorem}[section]
\newtheorem{lemma}{Lemma}[section]
\newtheorem{proposition}{Proposition}[section]

\newtheorem{corollary}{Corollary}[section]
\newtheorem{definition}{Definition}[section]
\newtheorem{remark}{Remark}[section]

\newcommand{\beqn}{\begin{eqnarray}}
	\newcommand{\eeqn}{\end{eqnarray}}
\newcommand{\beqs}{\begin{eqnarray*}}
	\newcommand{\eeqs}{\end{eqnarray*}}
\newcommand{\ban}{\begin{eqnarray*}}
	\newcommand{\nan}{\end{eqnarray*}}
\newcommand{\beq}{\begin{equation}}
	\newcommand{\eeq}{\end{equation}}

\newcommand{\eps}{\varepsilon}

\renewcommand{\det}{\mbox{det}}

\newcommand{\dist}{\mbox{dist}}

\newcommand{\R}{\mathbb{R}}

\numberwithin{equation}{section}

\numberwithin{equation}{section}
\numberwithin{figure}{section}

\begin{document}
	
	\title[$W^{2,1+\epsilon}$ estimates for potentials of optimal transport ]{\textbf{global $W^{2,1+\epsilon}$ regularity for potentials of optimal transport of non-convex planar domains}}
	
\author[S. Hu]
{Shengnan Hu}
\address{Shengnan Hu, MOE-LCSM, School of Mathematics and Statistics, Hunan Normal University, Changsha, Hunan 410081, P. R. China}
\email{helen@hunnu.edu.cn}

\author[Y. Li]{Yuanyuan Li}
\address{Yuanyuan Li, Institute for Theoretical Sciences, Westlake University, Hangzhou, 310030, China.}
\email{lyyuan@westlake.edu.cn}

\thanks{Research of the first author was supported by National Natural Science Foundation of China Youth Science Fund Project (Class C) (No. 12501260). Research of the second author was supported by China Postdoctoral Science Foundation (No. 2025M783144).
}	
	


\date{\today}

\dedicatory{}

	\begin{abstract} 
		In this paper, we investigate the optimal transport problem when the source is a non-convex polygonal domain in $\mathbb{R}^2$. We show a global $W^{2,1+\epsilon}$ estimate for potentials of optimal transport. Our method applies to a more general class of domains.
	\end{abstract}
	
	\maketitle
	
	\baselineskip=16.4pt
	\parskip=3pt
	
	\section{Introduction}
	Let $\Omega$ and $\Omega^*$ be bounded domains in $\mathbb{R}^2$. Let $f, g \in L^1(\mathbb{R}^2)$ represent two probability densities supported on $\Omega$ and $\Omega^*$, respectively, satisfying $\frac{1}{\lambda} \leq f, g \leq \lambda$ for a positive constant $\lambda$. We assume that $\Omega^*$ is convex. By Brenier's theorem \cite{Br1},
	there exists a globally Lipschitz convex function $u$ such that
	\begin{equation}\label{be1}
		(\nabla u)_\sharp f=g.
	\end{equation}
	Note that the optimal map is exactly given by $\nabla u.$

	
	The regularity of optimal transport maps has been extensively studied; see \cite{C92, C92b, C96, CL2, CLW1, CLW2, U1}. In the case of a non-convex source domain, when the densities are bounded between positive constants, the target domain is convex, and the source domain is a convex set with finitely many convex holes removed, Andriyanova and Chen provided global $C^{1,\alpha}$ estimates for the potentials of optimal transport maps in any dimension in \cite{AC}. Further assuming the smoothness of the densities and the regularity properties of the holes, Mooney and Rakshit proved a sharp global $W^{2,p}$ estimate in \cite{CR} in dimension two. Note that the positive distances from the holes to the boundary and between the holes are crucial in the above proof. 
	
	In our case, the source domain is a bounded non-convex polygon. Optimal transport between polygonal domains is of significant practical value. For instance, in numerical computation, particularly in finite element methods, generating high-quality meshes over polygonal domains is of paramount importance. Optimal transport can be employed to design mesh generation algorithms capable of distributing mass over polygonal domains, thereby enhancing simulation accuracy and efficiency. Readers are advised to refer to \cite{Levy,Mer} and the literature cited therein. Note that the global $C^1$ regularity always holds without any condition on $\Omega$, which is a classical result of Alexandrov; see also \cite{Fig4}. Our main result is:
	
	\begin{theorem}\label{t1}
		Let $\Omega$ be a bounded non-convex polygonal domain, $\Omega^*$ be a bounded convex domain in $\mathbb{R}^2$, and $u$ be an Alexandrov solution to \eqref{be1}. If $\frac{1}{\lambda} \leq f, g \leq \lambda$ within $\Omega$ and $\Omega^*$, respectively, for some positive constant $\lambda$, then $u\in W^{2, 1+\epsilon}(\overline{\Omega}).$
	\end{theorem}

	
	With the convexity of \(\Omega\), Caffarelli proved the doubling property of the Monge-Ampère measure associated with the potential function \(u\). In the case of a non-convex polygonal source domain, the doubling property may fail for a general convex set centered at some point in the closure of the domain. It is crucial to first establish the following doubling property for sections with small height.
	\begin{proposition}\label{doubling property}
		Under the same assumptions as in Theorem \ref{t1}, there exists a constant $h_0$ that depends on the geometry of $\Omega$ such that, for any $x \in \Omega$, we have
		$$|\nabla u(S^c_h(x))| \leq C |\nabla u\left(\frac{1}{2} S^c_h(x)\right)|,$$
		for some universal constant $C > 0$ and for all $h < h_0$.
	\end{proposition}

	The key to proving the above doubling property is to show that the sections are well localized. Specifically, any section centered in $\overline{\Omega}$ will be sufficiently small if its height is small. Once the sections are localized, the area of $S^c_h(x) \cap \Omega$ near the concave vertex is a positive portion of the area of $S^c_h(x)$, from which the doubling property follows.  
	
	Another issue to address is that near the concave vertex, it is not clear whether the engulfing property holds for all centered sections. Our strategy is to decompose the domain near this vertex into two sets and prove that the intersections of sections with each set satisfy the engulfing property.  
	
	The rest of the paper is organized as follows. In Section 2, we introduce some notation and preliminaries. Section 3 is devoted to the proof of the doubling property, specifically Proposition \ref{doubling property}. In the final section, we complete the proof of Theorem \ref{t1}.

	\section{Preliminaries}\label{S2}
	For the remainder of the paper, we will say that a constant is \textbf{universal} if it depends only on the dimension and $\lambda$, and is independent of the specific geometry of $\Omega$ or $\Omega^*$. We say that $a$ and $b$ satisfy $a \sim b$ if and only if there exist positive universal constants $c_1$ and $C_1$ such that $c_1 b \leq a \leq C_1 b$.  
	
	Suppose that $\Omega$ and $\Omega^*$ satisfy the hypotheses of Theorem \ref{t1}. Let $u$ be a globally Lipschitz convex solution to \eqref{be1} with $\nabla u(x) \in \Omega^*$ for almost every $x \in \mathbb{R}^2$. Similarly, as established in \cite{B}, there exists a globally Lipschitz convex function $v$ satisfying  
	\begin{equation}\label{bre2}
		(\nabla v)_\sharp g = f,
	\end{equation}
	with $\nabla v(y) \in \overline{\Omega}$ for almost every $y \in \mathbb{R}^2$. 
	
	Moreover, since the target domain $\Omega^*$ is convex, by \cite{C92}, if we extend $u$ to $\mathbb{R}^2$ via  
	\begin{equation}\label{extension}
		\tilde{u} := \sup \left\{L \mid L \text{ is linear, } L|_{\Omega} \leq u, \text{ and } L(x) = u(x) \text{ for } x \in \Omega \right\},
	\end{equation}
	then $\tilde{u}$ satisfies the Monge-Ampère inequality in the Alexandrov sense:  
	\begin{equation}\label{global eq}
		\frac{1}{C}\chi_\Omega \leq \det(D^2 \tilde{u}) \leq C\chi_\Omega \quad \text{in } \mathbb{R}^2,
	\end{equation} 
	where $C = \lambda^2$ is a universal constant.  
	We will continue to use $u$ to denote this extended function. Note that $u$ is $C^1$ in $\mathbb{R}^2$ and strictly convex in $\Omega$; see, for example, \cite{C93} or \cite[Lemma 2.3]{Mooney}.

	Denote $u^*$ as the Legendre transform of $u,$ namely, for any $y\in \mathbb{R}^2,$
	$$u^*(y):=\sup_{x\in\mathbb{R}^2}\{y\cdot x-u(x)\}.$$
	Note that $u^*=v$ in $\Omega^*.$
	\begin{definition}
		Given an open set $\Omega\subset \mathbb{R}^2$ and $w:\Omega\to \mathbb{R}$ a convex function, for $x\in \Omega,$ the subdifferential of $u$ at $x$ is defined as
		$$ \partial^{-}w(x):=\left\{y\in \mathbb{R}^2:w(z)\geq w(x)+y\cdot (z-x)\quad \forall z\in \Omega\right\}.
		$$
	\end{definition}
	It is immediate to check that $\partial^{-} w(x)$ is a  closed convex set. 
	
	Note that the subdifferentials of $u$ and $u^{*}$ are inverse to each other (seen in \cite{R}), that is
	\begin{equation}\label{eq}
		p\in \partial^{-}u(x)\Leftrightarrow x\in \partial^{-}u^*(p).
	\end{equation}
	
	\begin{definition}
		Given a point $x_0\in \mathbb{R}^2$ and a small positive constant $h$, we define the centred sub-level set of $u$ at $x_0$ with height $h$ as:
		\begin{equation}\label{sect}
			S^c_{h}(x_0) := \left\{x\in\mathbb{R}^2 : u(x)< u(x_0) + (x-x_0)\cdot \bar{p} + h\right\},
		\end{equation}
		where $\bar{p}\in \mathbb{R}^2$ is chosen such that the centre of mass of $S^c_{h}(x_0)$ is $x_0$. 
	\end{definition}
	By \cite[Lemma 2.1]{C96}, $S^c_h(x_0)$ is balanced with respect to $x_0,$ namely,
	$$x\in S^c_h(x_0) \Longrightarrow -Cx\in S^c_h(x_0)$$
	for some universal constant $C>0$ independent of $h.$
	\begin{remark}\label{center bound}
		If $u(0)=0$ and $u\geq 0,$ by \cite[Remark 2.2]{CLW3}, we have that $$0\leq u\leq Ch \text{  in   }S^c_h(0),$$
		for some universal constant $C>0.$ 
	\end{remark}
	In the following, we collect some important properties of sections used for the subsequent analysis. 
	\begin{lemma}[John's lemma \cite{J}]\label{John'sLemma}
		Let $S$ be a bounded convex subset $\R^2$ with its centre of mass at the origin. There is an ellipsoid $E$ with the origin as its center of mass such that $$E\subset S\subset  C_d E,$$ for some universal constant $C_d.$ 
	\end{lemma}
	For a set $S$ and a positive constant $\alpha,$ we say $\alpha S$ as the dilation of $S$ by a factor of $\alpha$ with respect to the center of $S.$ The original John's Lemma dose not require that the centre of mass is at the origin, and we can actually take $C_d=2^{3/2}.$

	\section{Doubling property for sections with properly choosing height}\label{S3}
	
	In this section we prove a localisation property and establish the doubling property for the Monge-Amp\`{e}re measure associated with $u.$
	
	Since $\Omega$ is a non-convex polygonal domain with the vertices given by $b_i\in\mathbb{R}^2$, $i=1,\cdots, m$, and the edges given by $[b_ib_{i+1}]$, $i=1, \cdots, m,$ with $b_{m+1}:=b_1$ and $b_0=b_m.$ Let $V$ be the set of the vertices of $\Omega,$ namely, $V=\left\{b_1, \cdots ,b_m\right\}.$ Let $V_1$ be the set of concave vertices and $V=V_1\cup V_2,$ where $V_2:=V\setminus V_1.$
	
	Given a closed convex set $\Sigma,$ we say $p\in \Sigma$ is an extreme point of $\Sigma$ if $p$ cannot be expressed as a convex combination $p=\lambda_0 p_0+(1-\lambda_0) p_1$ of points $p_0,p_1\in \Sigma$ with $\lambda_0\in (0,1)$ unless $p_0=p_1.$ Denote by $\text{ext}(\Sigma)$ the set of extreme points of $\Sigma.$ Similarly, a direction $q$ in the recession cone $\text{rc}(\Sigma):=\lim_{r\rightarrow 0} r \Sigma$ is extreme if $q=\lambda_0 q_0+(1-\lambda_0) q_1$ with $\lambda_0\in (0,1)$ forces $q_0$ to be nonnegative scalar multiple of $q_1$ or vice versa. A variation of Minkowski's theorem given by Rockafellar \cite[Theorem 18.5]{R} asserts that any closed convex set which does not contain a full line, can be expressed as the convex hull of its extreme points plus its extreme directions: $\Sigma =\text{conv} (\text{ext}(\Sigma) + \text{rc}(\Sigma)).$
	
	For any $x_0\in \partial \Omega,$ by a translation of coordinates, we may assume $x_0=0.$ By subtracting an affine function, we may assume $u\geq 0$ and $u(0)=0.$ Denote $\Sigma:=\left\{x\in \mathbb{R}^2: u(x)=0\right\},$ which is a closed convex set.
	By \cite{C92}, we have that $u$ is strictly convex in $\Omega,$ hence $\Sigma\cap \Omega=\emptyset.$
	Note also that $\text{ext} (\Sigma)\subset \partial \Omega,$ since otherwise 
	$\text{ext} (\Sigma)\cap(\mathbb{R}^2\setminus\overline{\Omega})\ne\emptyset,$
	which implies that the Monge-Amp\`ere measure $\det D^2u$ has positive mass in  
	$\mathbb{R}^2\setminus\overline{\Omega}$ contradicting to \eqref{global eq}.
	From equation \eqref{eq}, we have that $\Sigma=\partial^{-}u^*(0)$ since $u(x)=0$ holds if and only if $x\in \Sigma.$
	\begin{lemma}\label{extreme}
		Let $u,u^*$ and $\Sigma$ be as above. Then $0\in \text{ext}(\Sigma).$
	\end{lemma}
	\begin{proof}
		Since $u\geq 0$ and $u(0)=0,$ it follows that $0=\nabla u(0)\in \overline{\Omega}^*.$ Since $\Sigma=\partial^{-}u^*(0),$ it is equivalent to prove that $0$ is an extreme point of $\partial^{-} u^{*}(0).$ Note that $0\in \partial^{-}u^*(0).$ By subtracting a constant $u^*(0),$ we may assume that $u^*\geq 0.$
		
		By equation \eqref{global eq}, we have that for any small $r>0,$ $\nabla u(B_{r}(0))$ has positive Lebesgue measure. Hence, we can find a sequence $y_k\in \nabla u(B_{r}(0))$ such that $v$ is differentiable at $y_k$ with
		\begin{equation}\label{limit}
			y_k\in \Omega^*,\quad \nabla v(y_k)\rightarrow 0 \text{ and } y_k\rightarrow 0,
		\end{equation} 
		as $k\rightarrow \infty.$ It follows that $0\in \text{ext}(\partial^{-}v(0)).$ Similarly, we can deduce that
		\begin{equation}\label{a}
			\Sigma\cap \partial \Omega\subset \text{ext}(\partial^{-}v(0)),
		\end{equation}
		which implies that $\text{ext}(\partial^{-} u^{*}(0))\subset \text{ext}(\partial^{-}v(0))$ since $\text{ext} [\partial^{-}u^*(0)]=\text{ext} (\Sigma)\subset \partial \Omega.$
		
		Note that $u^*=v$ in $\Omega^*.$ If $0\in \Omega^*,$ we are done. So, in the following proof, we consider $0\in \partial \Omega^*.$ By \eqref{extension} and the definition of Legendre transform, it is easy to see that $u^*=\infty$ outside $\overline{\Omega}^*.$ Then $\partial^{-}u^*(0)$ is an unbounded, closed and convex set, which can be expressed as $\partial^{-}u^*(0) =\text{conv} (\text{ext}[\partial^{-}u^*(0)] + \text{rc}[\partial^{-}u^*(0)]).$ Note that Since $v\leq u^*$ and $v=u^*$ in $\overline{\Omega}^*,$ we can deduce that
		\begin{equation}\label{b}
			\partial^{-}v(0)\subset \partial^{-} u^{*}(0)=\Sigma.
		\end{equation}
		Since $\nabla v(y)\in \Omega$ for a.e. $ y\in \mathbb{R}^2,$ we can deduce that $\text{ext}(\partial^{-}v(0))\subset \overline{\Omega}.$ Combining this with \eqref{a} and \eqref{b}, it follows that
		\begin{equation}\label{c}
			\partial^{-}v(0)\cap \partial \Omega=\text{ext}(\partial^{-}v(0))=\Sigma\cap \partial \Omega.
		\end{equation}
		
		Assume by contradiction that assume $0\notin \text{ext}(\partial^{-}u^*(0)).$ Then we can express it as $0=p+r_0q$ for some $p\in \text{ext} [\partial^{-}u^*(0)],$ $q\in \text{rc}[\partial^{-}u^*(0)]$ and $r_0\in (0,\infty).$ This means that $u^*(y)\geq (p+rq)\cdot y$ for any $y\in \mathbb{R}^2$ and any sufficiently large $r>0.$ It follows that $q\cdot y\leq 0$ for $y\in \Omega^*.$ Additionally, we can rotate the supporting plane in the opposite direction $-q$ until the slope aligns with $p,$ namely, $u^*(y)\geq p\cdot y.$ Since $0=p+r_0q,$ it follows that $\Omega^*\subset \left\{y\in \mathbb{R}^2: p\cdot y\geq 0\right\}$ and $v(y)\geq p\cdot y$ in $\Omega^*.$ This leads a contradiction with \eqref{limit}. Indeed, we can find such $y_k\in \Omega^*$ that $\nabla v(y_k)\rightarrow 0$ as $k\rightarrow \infty.$
		
		Therefore, $0\in \text{ext}(\partial^{-}u^*(0))$ and $0\in \text{ext}(\Sigma).$
	\end{proof}
	
	Now, we are going to prove a localisation property for $x\in \overline{\Omega},$ which is crucial for the proof of the doubling property. 
	Observe that there is a constant $r_1>0,$ depending on the geometry of $\Omega,$ such that 
	\begin{equation}\label{r1}
		B_{3r_1}(b_i)\cap [b_jb_{j+1}]=\emptyset,\text{  for any }j\notin\left\{i-1,i\right\},
	\end{equation}
	for $i=1, \cdots, m.$
	Indeed, we only need to take 
	\begin{equation}\label{defbi}
		r_1:=\frac{1}{4}\min\left\{\dist(b_i,[b_jb_{j+1}]):b_i\notin [b_jb_{j+1}],1\leq i,j\leq m \right\}.
	\end{equation}
	By the choice of $r_1,$ we have that either $B_{r_1}(b_i)\cap\Omega$ is convex or $B_{r_1}(b_i)\setminus\Omega$ is convex; see Figure 3.1.
	
	\begin{figure}[htbp]
		\centering
		\includegraphics[width=4in,keepaspectratio]{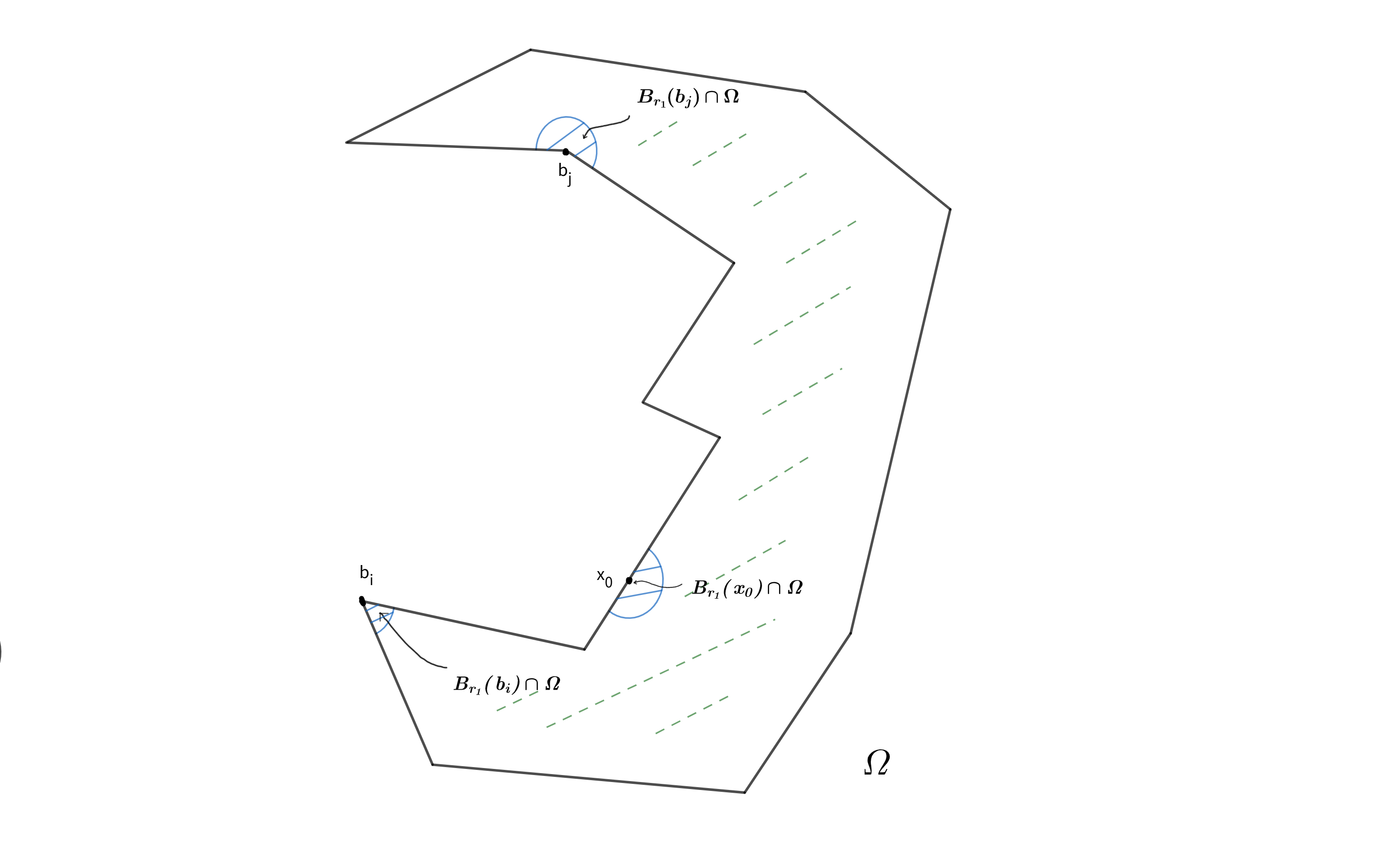}
		\caption {}
	\end{figure}
	
	\begin{lemma}\label{localisation1} For any given $r>0,$
		there exists $h_r>0$ such that for any $x\in \overline{\Omega},$  		
		$$S^c_h(x)\subset B_{r}(x),\text{ for }h\leq h_r.$$
	\end{lemma}
	\begin{proof}
		Suppose, for the sake of contradiction, that there exists $i\in\left\{1,\cdots,m\right\},$ $h_k\leq \frac{1}{k}$ and a sequence $x_k\in\overline{\Omega}$ such that
		$I_k\subset S^c_{h_k}(x_k)=\{u<L_k\}$ for some segment $I_k$ centered at $x_k$ and $|I_k|\geq c_1$ for some constant $c_1>0$ independent of $k,$ where $L_k$ is an affine function
		with $\|\nabla L_k\|\leq \|u\|_{L^\infty(\mathbb{R}^2)}.$ Up to a subsequence, we may assume 
		$x_k\rightarrow x_0\in \overline{\Omega},$ $L_k\rightarrow L_0, I_k\rightarrow I_0$ with $|I_0|\geq c_1.$ Note that $I_0$ is centered at $x_0.$
		
		By Remark \ref{center bound} we have that $u-L_k\geq -Ch_k$ in $\mathbb{R}^2.$ Passing to limit, we have $u\geq L_0$ in $\mathbb{R}^2.$
		Since $0\leq L_k-u\leq h_k$ in $I_k,$  it follows that $u=L_0$ on $I_0.$  Hence, $x_0$ is not an extreme point of $\left\{u=L_\infty\right\}.$ This contradicts the Lemma \ref{extreme}.
	\end{proof}

	Recall that $V_1$ denotes the set of concave vertices of $\Omega.$
	Since $\Omega$ is a polygonal domain, by Lemma \ref{localisation1}, there exists a small postive constant $h_0<h_{r_1},$ such that
	$S_h^c(x)\cap \Omega$ is convex for any $x\in \overline{\Omega}\setminus\left(\cup_{z\in V_1}B_{r_1}(z)\right)$ and any $h\leq h_0.$

	\begin{lemma}\label{dlem1}
		Let $h_0,r_1$ be given as above. For any $x\in \overline{\Omega},$ we have 
		$$|\nabla u(S^c_h(x))|\leq C|\nabla u(\frac{1}{2}(S^c_h(x)))|$$
		for some universal constant $C>0$ and for $h<h_0.$
	\end{lemma}
	\begin{proof}
		By the above discussion, if $x\in \overline{\Omega}\setminus\left(\cup_{z\in V_1}B_{r_1}(z)\right),$ we have that $S_h^c(x)\cap \Omega$ is convex, it implies
		\begin{equation}\label{p3}
			S^c_h(x)\cap \Omega\subset 2(\frac{1}{2}S^c_h(x)\cap \Omega)
		\end{equation}
		for $h<h_0.$
		Then the doubling estimate follows from \eqref{p3} and \eqref{global eq}.
		If $x\in B_{r_1}(b_i)$ for some $b_i\in V_1,$ then since $B_{2r_1}\setminus\Omega$ is convex and $S_h^c(x)\subset  B_{2r_1}(b_i)$ for $h\leq h_0,$ it implies that 
		\begin{equation}\label{p2}
			|S^c_h(x)\cap \Omega|\approx |S^c_h(x)|\approx|\frac{1}{2}S^c_h(x)\cap \Omega|
		\end{equation}
		for $h<h_0.$ 
		Then the doubling estimate follows from \eqref{p2} and \eqref{global eq}.
	\end{proof}

	Once we have the doubling property, by the same proof of \cite[Lemma 2.4 and Corollary 2.1]{C96} we have the following properties.
	\begin{corollary}\label{prop}
		For any centered section $S_h^c(x)$ with $x\in \overline{\Omega}$ and $h\leq h_0$ defined by an affine function $L$,
		\begin{itemize}
			\item[(a)] $|S_h^c(x)||S_h^c(x)\cap\Omega|\approx h^2.$
			
			\item[(b)] $\sup_{S_h^c(x)}|u-L|\approx (L-u)(x)=h.$
			
			\item[(c)] if $x+z\in \partial S_h^c(x),$ then $(L-u)(x+tz)\leq C(1-t)^{1/2}h$ for $t\in [0,1].$
			
		\end{itemize}
	\end{corollary}
		

	\section{Geometric properties of sections near concave vertices}
	
	Let $h_0, r_1$ be as in Lemma \ref{dlem1}. For any $x_0\in \Omega_1:=\overline{\Omega}\setminus\left(\cup_{z\in V_1}B_{r_1}(z)\right),$ by the choice of $h_0$ we have that 
	$S_h^c(x_0)\cap\Omega$ is a convex set for any $h\in (0, h_0).$ 
	By \eqref{global eq}, we have that 
	\begin{equation}\label{loceq}
		\det D^2u=g\chi_{_{S_h^c(x_0)\cap\Omega}}
	\end{equation}
	in 
	$S_h^c(x_0)$ for some function $g$ satisfying $1/C\leq g\leq C$ in $S_h^c(x_0).$
	By \eqref{loceq}, since $S_h^c(x_0)\cap\Omega$ is convex the sections are well localised
	(see Lemma \ref{localisation1}),  the proof in \cite{SY2} applies in our case and we have that
	$\|u\|_{W^{2,1+\epsilon}(\frac{1}{2}S_h^c(x_0)\cap\Omega)}\leq C.$ Then, by a standard covering argument we have
	\begin{equation}
		\|u\|_{W^{2,1+\epsilon}(\Omega_1)}\leq C,
	\end{equation}
	for some positive constant $C$ depending on $\lambda$, the geometry of $\Omega$, and the inner and outer radii of $\Omega^*$.
	Hence to prove Theorem \ref{t1} we only need to show that
	$\|u\|_{W^{2,1+\epsilon}(B_{r_1}(b_i)\cap \Omega)}\leq C$ for $b_i\in V_1,$ namely $b_i$ is a concave vertex of $\Omega.$

	Now, we will establish a geometric decay property of sections near the concave vertices.
	In the following we denote $[S^*,\bar{u}]$ as a renormalization pair of $[S_h^c(x), u].$
	where $B_1(0)\subset S^*:=T(S_h^c(x))\subset B_{C_d}(0)$ is centered at $0,$  $T$ is an symmetric affine transformation, and $\bar{u}(x):=\frac{1}{h}(u-L)(T^{-1}x).$
	\begin{lemma}[Geometric Decay]\label{geodecay}
		Let $x_1,x_2$ be two points in $B_{r_1}(b_i)\cap \overline{\Omega},$ and let $0\leq t<\bar{t}\leq 1.$ There exists an $s_0(t,\bar{t})\in (0,1)$ such that $x_2\in t S^c_h(x_1)\cap \overline{\Omega}$ implies $S^c_{sh}(x_2)\subset \bar{t} S^c_h(x_1)$ for all $s\leq s_0.$
	\end{lemma}
	\begin{proof}
		We will prove that given $t<1$ and $x_2\in t S^c_h(x_1)\cap \overline{\Omega},$ there exist $s_0(t)$ and $\bar{t}<1$ such that
		$$
		S^c_{s h}\left(x_2\right) \subset \bar{t} S^c_h\left(x_1\right),
		$$
		for $s\leq s_0.$
		The lemma will then follow for a general $\bar{t}>t$ by a finite iteration.
		
		Assume for contradiction that the lemma fails. Then for a fixed $t$, we have a sequence of solutions $u_k$, domains $\Omega_k$, pairs of points $x_1^k, x_2^k$ in $\overline{\Omega_k}$, and sections $
		S^c_{h_k}\left(x_1^k\right), S^c_{\frac{1}{k} h_k}\left(x_2^k\right)
		$ such that $x_2^k \in t S^c_{h_k}\left(x_1^k\right)$, but $S^c_{\frac{h_k}{k}}\left(x_2^k\right)$ is not contained in $(1-1 / k) S^c_{h_k}\left(x_1^k\right)$. 
		We renormalize $\left[S^c_{h_k}\left(x_1^k\right), u\right]$ to $\left[S_k^*, \bar{u}_k\right]$ with $S_k^*$ equivalent to the unit ball. It is straightforward to compute that 
		$$\det D^2 \bar{u}_k=\frac{1}{|S_k^*\cap T_k(\Omega_k)|}\tilde{f}\chi_{S_k^*\cap T_k(\Omega_k)} \quad \text{ in }S^*_k,$$ where $\frac{1}{\lambda}\leq \tilde{f}\leq \lambda$ and $T_k$ is the renormalization. Note that $S_k^*\setminus \overline{T_k(\Omega_k)}$ is convex and $|S_k^*\cap T_k(\Omega_k)|\approx 1.$
		By Blaschke's selection theorem and Corollary \ref{prop}, up to a subsequence we may assume
		$\bar u_k\rightarrow u_\infty, S^*_k\rightarrow S_\infty, T_k(\Omega)\rightarrow \Omega_\infty$ as $k\rightarrow \infty.$ Moreover, $$\det D^2 u_{\infty}\approx \chi_{S_{\infty}\cap \Omega_{\infty}}$$ in the Alexandrov sense. Note that $S_\infty\setminus\overline{\Omega_\infty}$ is convex.

		Denote the new sections of $S_{\frac{1}{k} h_k}\left(x_2^k\right)$ after these renormalizations by $S^*_{\frac{1}{k}}(x^*_k)$ with new centers at $x^*_k,$ whose height tends to zero as $k\rightarrow \infty.$ Therefore, in the limit, we obtain that $u_{\infty}$ coincides with a 
		supporting affine function $L_{\infty}$ in a convex set $K.$ Since $S^*_{\frac{1}{k}}(x^*_k)$ extends beyond $(1-1/k)S_k^*$ from $x^*_k\in tS_k^*,$ $K$ must contain a segment from $x^*_{\infty}\in t S_{\infty}$ to $x_{\infty}\in \partial S_{\infty},$ where $x^*_{\infty}$ and $x_{\infty}$ are subsequential limits of the $x^*_k$ and points $x_k\in S^*_{\frac{1}{k}}(x^*_k)\cap \partial \left((1-1/k)S_k^*\right),$ respectively. Since $S^*_{\frac{1}{k}}(x^*_k)$ are balanced around $x^*_k,$ there exists $\tilde{x}_{\infty}\in K$ such that $x^*_{\infty}=ax_{\infty}+(1-a)\tilde{x}_{\infty}$ for  some constant $0<a<1.$
		
		Since
		\begin{equation}\label{uinfeq}
			\det D^2 u_{\infty}\gtrsim \chi_{S_{\infty}\cap \Omega_{\infty}}\quad \text{ in }S_{\infty}\subset \mathbb{R}^2,
		\end{equation} 
		it follows that $u_{\infty}$ is strictly convex inside $S_{\infty}\cap \Omega_{\infty}.$
		Hence $(\tilde{x}_{\infty}x_{\infty})\cap \Omega_\infty=\emptyset.$
		
		We can select $p_1,p_2\in [x_{\infty}\tilde{x}_{\infty}]$ such that $|p_1x^*_{\infty}|=|p_2x^*_{\infty}|=r$ and $B_r(x^*_{\infty})\subset S_{\infty}$ for some constant $r$ only depending on $t.$ Note that $S_{\infty}\setminus \overline{\Omega_{\infty}}$ is convex. Combining this with the fact that $x^*_{\infty}\in \overline{\Omega_{\infty}}$ and $[p_1p_2]\cap \Omega_{\infty}=\emptyset,$ we can deduce that $x^*_{\infty}\in \partial \Omega_{\infty}$ and $B_r(x^*_{\infty})\cap \Omega_{\infty}$ is a semiball, denoted by $B^+_r(x^*_{\infty})$ with its straight boundary $[p_1p_2]\subset \partial \Omega_{\infty}.$ 
		Let $e^+$ be a unit vector orthogonal to $[p_1p_2]$ and point into the interior of  $\Omega_{\infty}.$
		Now, $u_{\infty}$  satisfies
		\begin{equation}\label{nonexieq}
			\det D^2 u_{\infty}\gtrsim \chi_{\left\{(x-x^*_{\infty})\cdot e^+>0\right\}} ,\qquad u_{\infty}|_{\left\{(x-x^*_{\infty})\cdot e^+=0\right\}}\text{ linear},
		\end{equation}
		in $B_r(x^*_{\infty}).$
		However, it is proved in the Lemma 6.1 in \cite{CR} that there is no convex function on $B_r(x^*_{\infty})\subset \mathbb{R}^2$ satisfying \eqref{nonexieq}, a contradiction which completes the proof.
	\end{proof}
	For any $x\in B_{r_1}(b_i),$ by Lemma \ref{localisation1} and  the definition of $r_1$ (\eqref{defbi}), we have that the relation \eqref{p2} holds. Hence, by Corollary \ref{prop} (a), we have
	\begin{equation}\label{volu1}
		|S^c_h(x)\cap \Omega|\approx |S^c_h(x)|\approx h.
	\end{equation}
	Once we have the geometric decay property, by the same proof of \cite[Corollary 2.2]{C96} we have the following property.
	\begin{corollary}\label{nonconvpart}
		Suppose $x\in \overline{B_{r_1}(b_i)\cap \Omega}$ for some $b_i\in \Sigma_1.$ Then,
		$\nabla u(S^c_h(x))$ is conjugate to $S^c_h(x),$ namely, let $A$ be an affine transformation such that
		$A\left(S^c_h(x)\right)\sim B_{h^{\frac{1}{2}}}(0)$, then $A'^{-1} \left(\nabla u(S^c_h(x))\right)\sim B_{h^{\frac{1}{2}}}(p)$ for some point $p,$
		where $A'$ is the transpose of $A$;
	\end{corollary}

	
	
	
	Denote the inner normal vectors of $(b_{i-1}b_i)$ and $(b_{i}b_{i+1})$ by $e^{+}$ and $e^{-}$, respectively. 
	Let 
	\begin{equation}\label{halfball}
		B^{\pm}_{r_1}(b_i):=\overline{B_{r_1}(b_i)}\cap\left\{y\in \mathbb{R}^2 : e^{\pm}\cdot (y-b_i)\geq 0\right\}
		\subset \overline{B_{r_1}(b_i)}\cap \overline{\Omega}.
	\end{equation}
	Denote 
	\begin{equation}\label{halfsection}
		S^{\pm}_h(x):=S^c_h(x)\cap \left\{y\in \mathbb{R}^2: (y-b_i)\cdot e^{\pm}\geq 0\right\}.
	\end{equation} 
	For $x\in B^{\pm}_{r_1}(b_i),$ by Lemma \ref{localisation1} and the definition of $r_1$ \eqref{defbi} 
	we have that $S^{\pm}_h(x)\subset \overline{\Omega}.$ 
	
	Then, we have the following engulfing property.
	
	\begin{lemma}[Engulfing property]\label{engulfing2}
		There exists positive constants $h_0$ (smaller than that in Lemma \ref{dlem1}) and $C>1$ such that for any
		$x\in B^{+}_{r_1}(b_i),$ $y\in S^{+}_h(x),$ and $h\leq h_0$, we have $S^{+}_h(x)\subset S^{+}_{Ch}(y).$
	\end{lemma}
	\begin{proof}
		The proof is similar to that of Lemma \ref{geodecay}. 
		Suppose the the engulfing property fails. Then there exists a sequence $h_k\leq \frac{1}{k},$
		$x_k\in B^{+}_{r_1}(b_i),$ $y_k\in S^{+}_{h_k}(x_k),$ such that 
		$S^{+}_{h_k}(x_k)\subset S^{+}_{kh_k}(y_k)$ fails.
		We renormalize $\left[S^c_{kh_k}(y_k), u\right]$ to $\left[S_k^*, \bar{u}_k\right]$ with $S_k^*=T_kS^c_{kh_k}(y_k)$ equivalent to the unit ball. 
		Up to a subsequence we may assume
		$\bar u_k\rightarrow u_\infty, S^*_k\rightarrow S_\infty, T_k(\Omega)\rightarrow \Omega_\infty$ as $k\rightarrow \infty.$
		Let  $S^*_{\frac{1}{k}}(x^*_k)$ be the renormalization of
		$S^c_{h_k}\left(x_1^k\right).$ Note that its height tends to zero as $k\rightarrow \infty.$ Therefore, in the limit, we obtain that $u_{\infty}$ coincides with a 
		supporting affine function $L_{\infty}$ in a convex set $K.$ Note that $0\in K\cap\overline{\Omega_\infty}.$
		Since $S^{+}_{h_k}(x_k)\subset S^{+}_{kh_k}(y_k)$ fails, we have that
		$S^{+}_{h_k}(x_k)\cap \partial S^{+}_{kh_k}(y_k)\cap \overline{\Omega}\ne \emptyset.$ 	
		Passing to limit, 
		we have that $K\cap\partial S_\infty\cap \overline{\Omega_\infty}\ne\emptyset. $
		Pick a point $x'\in K\cap\partial S_\infty\cap \overline{\Omega_\infty}.$
		Let $e$ be a unit vector orthogonal to $[0x']$ and point into the interior of  $\Omega_{\infty}.$
		Let $x_\infty=\frac{1}{2}(x'+0).$ Since $0, x'\in \overline{\Omega_\infty}\cap K,$ 
		there exists $r>0$ small such that $B_r(x_\infty)\cap \{x\cdot e\geq 0\}\subset \overline{\Omega_\infty}.$
		Now, similar to the proof of Lemma \ref{geodecay}, we have $u_{\infty}$  satisfies
		\begin{equation}\label{nonexieq}
			\det D^2 u_{\infty}\gtrsim \chi_{\left\{(x-x_{\infty})\cdot e^+>0\right\}} ,\qquad u_{\infty}|_{\left\{(x-x_{\infty})\cdot e^+=0\right\}}\text{ linear},
		\end{equation}
		in $B_r(x_{\infty}).$ 
		This contradicts to \cite[Lemma 6.1]{CR}.
	\end{proof}

	A direct consequence of the above engulfing property is the following Vitali covering lemma.
	\begin{lemma}\label{cover2}
		Let $D$ be a compact subset of $B^{+}_{r_1}(b_i)$ and let $\{S^+_{h_x}(x)\}_{x\in D}$ be a family of sets defined in \eqref{halfsection} with $h_x\leq h_0.$ Then, there exists a finite subfamily $\{S^+_{h_{x_j}}(x_j)\}_{j=1,\cdots, l}$ such that 
		$$D\subset \bigcup^{l}_{j=1} S^+_{h_{x_j}}(x_j)$$
		with $\left\{S^+_{\eta_0 h_{x_j}}(x_j)\right\}_{\left\{j=1,\cdots, l\right\}}$ disjoint, where $\eta_0>0$ is a universal constant.
	\end{lemma}
	Note that the above lemma also holds for $B^{+}_{r_1}(b_i), \{S^+_{h_x}(x)\}_{x\in D}$ replaced
	by $B^{-}_{r_1}(b_i), \{S^-_{h_x}(x)\}_{x\in D}$
	
	
	\section{Proof of Theorem \ref{t1}}
	
	In this section, we prove the global $W^{2,1+\epsilon}$ estimate of $u$ near the concave vertices, and thus complete the proof of Theorem \ref{t1}.
	The proof mainly follows the methods introduced in
	De Philippis-Figalli \cite{DF} and De Philippis-Figalli-Savin \cite{DFS}, see also  Chen-Liu-Wang \cite{CLW2} and Savin-Yu \cite{SY2} for the boundary case.
	Let $b_i\in V_1$ be a concave vertex of $\Omega.$ Let $h_0, r_1, B^{\pm}_{r_1}(b_i), S_h^{\pm}$ be as in \eqref{defbi}, Lemma \ref{engulfing2}, \eqref{halfball} and \eqref{halfsection} respectively. 
	All we need to show is that $\int_{B^{\pm}_{r_1}(b_i)}|D^2u|^{1+\epsilon}\leq C.$
	In the following we assume $u\in C^2(\Omega)$  and the estimate for general solutions can be deduced via a standard approximation argument.
	
	Given $x_0\in B^{\pm}_{r_1}(b_i),$  $S^c_h(x_0)=\left\{x\in \mathbb{R}^2|u(x)<\ell\right\}$ for some affine function $\ell.$ 
	We have $u=\ell$ on $\partial S^c_h(x_0)$ and $\ell(x_0)=u(x_0)+h.$
	Let $A$ be the affine transformation such that $A\left(S^c_h(x_0)\right)\approx B_{h^{\frac{1}{2}}}$ with $\det\, A=1.$ As in \cite{DFS} we define ${\bf a}(S^c_h(x_0)):=\|A\|^2$ as the normalised size of $S^c_h(x_0).$ 
	
	\begin{proposition}\label{decay} 
		There exists a universal constant $C_0>0$ such that
		for any $x_0\in \Omega,$ we have
		\begin{equation}\label{decay1}
			\int_{S^+_h(x_0)\cap B^{+}_{r_1}(b_i)}\|D^2u\|\leq C_0{\bf a} \big|\{C_0^{-1}{\bf a}\leq \|D^2u\|\leq C_0{\bf a}\}\cap S^+_{\eta_0 h}(x_0)\cap\Omega\big|,
		\end{equation}
		where ${\bf a}:={\bf a}(S^c_h(x_0)).$
	\end{proposition}
	
	\begin{proof}
		
		Let $A$ be the affine transformation such that $A\left(S^c_h(x_0)\right)\approx B_{h^{\frac{1}{2}}}(0)$ with $\det\, A=1.$ 
		Let $\bar{u}(x):=(u-\ell)(A^{-1}x).$ It is straightforward to check that $\det D^2\bar{u}=\tilde{f},$ where $C^{-1}<\tilde{f}(x)<C.$
		By  Corollary \ref{nonconvpart},  we have $|\nabla\bar{u}(x)|\lesssim h^{\frac{1}{2}}$ for any $x\in  A(S^c_h(x_0)).$ 
		Note that $|\partial (AS^c_h(x_0))|\approx h^{\frac{1}{2}}.$
		Therefore, 
		\begin{equation}\label{es11}
			\int_{S^+_h(x_0)\cap B^{+}_{r_1}(b_i)}\|D^2u\|\leq {\bf a}\int_{A(S^+_h(x_0))}\|D^2\bar{u}\|={\bf a}\int_{\partial (A(S^+_h(x_0)))}\bar{u}_\nu\approx {\bf a} h\leq C_1{\bf a}|S^+_{\eta_0 h}(x_0)\cap B^{+}_{r_1}(b_i)|,
		\end{equation}
		where $C_1$ is a universal constant.
		
		From \eqref{es11}, it follows that 
		\begin{equation}\label{es12}
			\big|\{\|D^2\bar{u}\|\leq 2C_1\}\cap A(S^+_{\eta_0 h}(x_0)\cap B^{+}_{r_1}(b_i))\big|\geq \frac{1}{2}|S^+_{\eta_0 h}\cap B^{+}_{r_1}(b_i)|.
		\end{equation}
	
	On the other hand, since $C^{-1}<\det D^2\bar{u}<C$ in $A\Omega,$ we have 
	\begin{equation}\label{es13}
		C_0^{-1} I\leq D^2\bar{u}<C_0I\qquad \text{ in}\quad \{\|D^2\bar{u}\|\leq 2C_1\}\cap A (B^{+}_{r_1}(b_i)),
	\end{equation}
	for some universal constant $C_0.$ By the definition of $\bar{u},$ it is straightforward to check that
	\begin{equation}\label{es14}
		\{C_0^{-1}I\leq D^2\bar{u}<C_0I\}\cap A (B^{+}_{r_1}(b_i)) \subset A\{C_0^{-1}{\bf a}\leq \|D^2u\|\leq C_0{\bf a}\}\cap A (B^{+}_{r_1}(b_i)).
	\end{equation}
	Combining \eqref{es11}--\eqref{es14}, we obtain the desired result in equation \eqref{decay1}.
\end{proof}

Let us define ${\bf a_h}$ to be the normalized size ${\bf a}$ as in Lemma \ref{decay} for the section $S^c_h(x_0).$ Let $N>0$ be a constant to be determined later. For each integer $m$, let $$D^+_m:=\{x\in B^{+}_{r_1}(b_i)\cap \Omega: |D^2u(x)|\ge N^m\}.$$ 

\begin{lemma}\label{LevelSetDecay}
	Let $u$ be a globally Lipschitz convex solution to \eqref{be1} subject to boundary condition $\nabla u(\Omega) =\Omega^*.$ There is a universal constant $\tau\in(0,1)$ such that    $$\int_{D^+_{m+1}}|D^2u|\le (1-\tau)\int_{D^+_m}|D^2u|.$$
\end{lemma} 

\begin{proof}
	
	For $x\in D^+_{m+1}$, $|D^2u(x)|\ge N^{m+1}$, we know that ${\bf a_h}$ ranges from $1/C_1$ to $N^{m+1}/C_2$ as $h$ changes from $h_0$ to $0$ for some universal constants $C_1$ and $C_2.$
	
	Thus, we let $N :=\max\{C_0C_2,C_0^2\}$ and we have $C_0 N^{m}\le N^{m+1}/C_2$. Hence we can take $0<h_x\leq \frac{h_0}{\bar{C}}$ so that $$C_0N^m\le {\bf a_h}< C_0^{-1}N^{m+1}.$$ 
	
	Let $\mathcal{F}$ be the set of sections $\left\{S^c_{h_x}(x)\right\}_{x\in D^+_{m+1}}.$ By Lemma \ref{cover2}, we can select a finite subfamily $\{S^+_{h_j}(x_j)\}_{ \left\{j=1,\cdots,l\right\}}$ such that
	\begin{equation}\label{Cov}
		D^+_{m+1}\subset \bigcup^{l}_{i=1} S^+_{h_i}(x_i).
	\end{equation}
	For each section $S^c_{h_x}(x)\in \mathcal{F},$
	Proposition \ref{decay} implies $$\int_{S^+_h(x_0) \cap  B^{+}_{r_1}(b_i)}\|D^2u\| \leq C_0 {\bf a_h} \big|\{ C_0^{-1}{\bf a_h}\leq \|D^2u\| \leq C_0 {\bf a_h}\} \cap S^+_{ \eta_0 h}(x_0) \cap  B^{+}_{r_1}(b_i) \big|.$$ Moreover, by our choice of $h_x,$  we have $C_0N^m\le {\bf a_h}< C_0^{-1}N^{m+1}$ and we can conclude that
	\begin{equation*}
		\int_{S^+_h(x)\cap B^{+}_{r_1}(b_i)}|D^2u|\le C\int_{S^c_{\eta_0h(x)}\cap (D^+_m\setminus D^+_{m+1})}|D^2u|.
	\end{equation*} 
	
	Now, we can combine the above results with \eqref{Cov} to deduce the following:
	\begin{align*}
		\int_{D^+_{m+1}}|D^2u|&\le\sum\int_{S^+_{h_j}(x_j)\cap B^{+}_{r_1}(b_i)}|D^2u|\\&\le \sum C\int_{S^+_{\eta_0h_j}(x_j)\cap (D^+_m\setminus D^+_{m+1})}|D^2u|\\&=C\int_{D^+_m\setminus D^+_{m+1}}|D^2u|\sum\chi_{S^+_{\eta_0h_j}(x_j)}\\&\le C l\int_{D^+_m\setminus D^+_{m+1}}|D^2u|\\&=C l\int_{D^+_m}|D^2u|-C l\int_{D^+_{m+1}}|D^2u|.
	\end{align*}
	Therefore, we conclude that
	$$\int_{D^+_{m+1}}|D^2u|\le \frac{C l}{1+C l}\int_{D^+_{m}}|D^2u|,$$ where $C$ is a universal constant.

\end{proof} 

Once the decay of integrals over $D^+_m$ is established, the $W^{2,1+\eps}$-estimate follows directly from a standard iteration argument.
\begin{proof}[Proof of Theorem \ref{t1}]
	Let $N$ and $\tau$ be as in Lemma \ref{LevelSetDecay}. 
	
	Fix an integer $m_0.$ For $t>N^{m_0}$, there is an integer $k$ such that $$N^{m_0+k}\le t< N^{m_0+k+1},$$ that is, $k\le \log_N t-m_0<k+1.$
	
	By Lemma \ref{LevelSetDecay}, an iterative argument leads to 
	\begin{align*}\int_{\{|D^2u|\ge t\}\cap B^{+}_{r_1}(b_i)}|D^2u|&\le  \int_{D^+_{m_0+k}}|D^2u|\\&\le (1-\tau)^{k}\int_{D^+_{m_0}}|D^2u|\\&\le (1-\tau)^{-1-m_0}\cdot t^{\log_N(1-\tau)}\cdot\int_{D^+_{m_0}}|D^2u|\\&\leq Ct^{-\eps_0},
	\end{align*}
	where $\eps_0:=-\log_N(1-\tau)>0$ is universal, and $C$ depends on the inner and outer radii of $\Omega$ and $\Omega^*,$ and we use $\int_{D^+_{m_0}}|D^2u|\le\int_{ \Omega}\Delta u=\int_{\partial\Omega}\nabla u\cdot \nu\le C$ in the last inequality. 
	
	By Markov's inequality, this gives $|\{|D^2u|\ge t\}\cap B^{+}_{r_1}(b_i)|\le Ct^{-1-\eps_0}$ for $t\ge N^{m_0}$.
	
	Therefore, for any $\eps\in(0,\eps_0),$ it follows 
	\begin{align*}
		\int_{B^{+}_{r_1}(b_i)} |D^2u|^{1+\eps}&=\int_{D^+_{m_0}}|D^2u|^{1+\eps}+\int_{\{|D^2u|\le N^{m_0}\}\cap B^{+}_{r_1}(b_i)}|D^2u|^{1+\eps}\\&\le C_\eps\int_{N^{m_0}}^\infty t^{\eps}|\{|D^2u|\ge t\}\cap B^{+}_{r_1}(b_i)|dt+N^{m_0(1+\eps)}| B^{+}_{r_1}(b_i)|\\&\le C\int_{N^{m_0}}^\infty t^{-1-\eps_0+\eps}dt+N^{m_0(1+\eps)}| B^{+}_{r_1}(b_i)|\\&=CN^{m_0(\eps-\eps_0)}+N^{m_0(1+\eps)}| B^{+}_{r_1}(b_i)|.
	\end{align*}
	Hence, we have proved that
	$$\|u\|_{W^{2,1+\epsilon}(B^+_{r_1}(b_i)\cap \Omega)}\leq C^+.$$
	Note that the above proof also holds for $B^{+}_{r_1}(b_i), D^+_m$ replaced by $B^{-}_{r_1}(b_i), D^-_m.$
	
	Therefore, for any $\eps\in(0,\eps_0),$ it follows 
	\begin{align*}
		\int_{\Omega} |D^2u|^{1+\eps}&=\int_{\Omega_1}|D^2u|^{1+\eps}+\sum_{b_i\in V_1}\int_{\Omega\cap B_{r_1}(b_i)}|D^2u|^{1+\eps}\\&\le C + \sum_{b_i\in V_1}\int_{\Omega\cap B^{+}_{r_1}(b_i)}|D^2u|^{1+\eps} + \sum_{b_i\in V_1}\int_{\Omega\cap B^{-}_{r_1}(b_i)}|D^2u|^{1+\eps}\\&\le C + C^+ +C^-.
	\end{align*}
\end{proof}

\bibliographystyle{amsplain}

\end{document}